\documentclass[a4]{amsart}
\usepackage{ifthen}
\usepackage{hyperref}
\usepackage{amssymb}
\usepackage{amsmath}
\usepackage{amsfonts}
\usepackage{amsthm}
\usepackage{euscript}
\usepackage[rotate,all,color]{xy}
\usepackage[figuresright]{rotating}
\usepackage{graphicx}
\usepackage{color}
\usepackage[utf8]{inputenc}
\usepackage[T1]{fontenc}
\usepackage{todonotes}
\usepackage[capitalize]{cleveref}
\usepackage{verbatim}
\usepackage{natbib}

\newcommand{\defterm}[1]{\textbf{#1}}
\renewcommand{\phi}{\varphi}
\newcommand{\C}{\mathbb{C}}
\DeclareMathOperator{\from}{:}
\renewcommand{\to}{\operatorname{\longrightarrow}}

\renewcommand{\mapsto}{\operatorname{\longmapsto}}

\newcommand{\curly}[1]{\left\{#1\right\}}

\newcommand{\id}{\mathrm{id}}
\newcommand{\Id}{\mathrm{Id}}

\newcommand{\ladj}{\dashv}

\newcommand{\op}{\mathrm{op}}

\newcommand{\catfont}{\mathsf}
\newcommand{\Top}{\catfont{Top}}
\newcommand{\Cat}{\catfont{Cat}}
\newcommand{\CAT}{\catfont{CAT}}
\newcommand{\SSet}{\catfont{SSet}}
\newcommand{\Set}{\catfont{Set}}

\newtheorem{thm}{Theorem}

\newtheorem{prp}[thm]{Proposition}
\newtheorem{dfn}[thm]{Definition}

\newtheorem{rem}[thm]{Remark}

\title{A Categorical Model for the Hopf Fibration}

\author{Björn Gohla}
\email[Björn Gohla]{b.gohla@gmx.de}
\address{Grupo de F\'isica Matem\'atica, Faculdade de Ci\^ encias da 
Universidade de Lisboa, Campo Grande, Edif\'icio C6, 1749-016  Lisboa, Portugal}

\thanks{BG was supported by the Portuguese Science Foundation (FCT)
  under the post-doc grant SFRH/BPD/99060/2013.}

\begin{document}

\begin{abstract}
  \noindent
  We give a description up to homeomorphism of $S^3$ and $S^2$ as
  classifying spaces of small categories, such that the Hopf map
  $S^3\to{}S^2$ is the realization of a functor.
\end{abstract}

\maketitle{}

\section{Introduction}
\noindent
\citet{zbMATH01547910} describe simplicial complices whose realizations
are $S^3$ and $S^2$ such that the Hopf map is the realization of a
simplicial map. \citet{zbMATH03457106} Introduced the classifying
space of a category in order to study algebraic $K$-theory. In
particular he defined homotopy groups of (additive) categories as
homotopy groups of their nerves.

We on the other hand were interested in seeing if one can engineer
categories whose nerves would have interesting an geometry. The first
thing one notices is that such categories should not contain
non-trivial endomorphisms or isomorphisms, since this would
immediately make the nerve very large, and arguably ungeometric.

So one is led to focus on the following class of categories and their functors:
\begin{dfn}
  A category \(A\) is called \defterm{skeletal} if for all arrows
  \(\xymatrix@1{x\ar[r]^f&y}\) in \(A\) we have that \(f\) being an
  isomorphism implies \(x=y\).
\end{dfn}
Note that this does not mean that \(f\) is an identity.
\begin{dfn}
  A category \(A\) is called \defterm{progressive} if all
  endomorphisms \(\xymatrix@1{x\ar[r]^f&x}\) in \(A\) are identities.
\end{dfn}
A category that is both skeletal and progressive is in some ways
similar to a poset. These categories were called \defterm{acyclic} and
studied by \citet{zbMATH05175082}, in particular their quotients with
respect to group actions. We will only have reason to use finite such
categories.

We would like to note that the existence of a categorical model of the
Hopf map itself is clear, since one can take the simplicial model
$\eta\from{}S^3_{12}\to{}S^2_4$ of \citet{zbMATH01547910} and consider
the posets of simplices ordered by inclusion, and the induced order
preserving map between them
$U(\eta)\from{}U(S^3_{12})\to{}U(S^2_4)$. These posets are by
implication acyclic categories. The realization has the right
homeomorphism type, but $U(S^3_{12})$ is a category with 168 objects,
$U(S^2_4)$ with 14. By comparison, the acyclic categories in our
construction have 10 and 4 objects. This seems to be a minimum.

Finally we remark, that one could define a categorical model $R$ of
$\C{}P^2$ from our model $H$ of the Hopf map by forming the pushout
\begin{equation}
  \label{eq:1}
  \begin{xy}
    *!C\xybox{
      \xymatrix{
        P\ar[d]_H\ar[r]^-{i_0}&M_{P\to{}\underline{0}}\ar[d]\\
        Q\ar[r]&R
      }
    }    
  \end{xy}
\end{equation}
where $M_{P\to\underline{0}}$ is the cone over $P$, i.e., a model of
the 4-ball. We did not pursue this avenue any further.

\section{From Categories to Spaces}
\label{sec:from-categ-spac}
\noindent
Let $\Delta_+$ be the category of finite non-empty ordinals, i.e.,
categories of the form $\underline{n}=\{0\to{}1\to\cdots\to{}n\}$, and
the order preserving maps between them, i.e., functors. $\Delta_+$ is
obviously a subcategory of the category $\Cat$ of all small categories
by means of the functor $I\from\Delta_+\to\Cat$. 

Simplicial sets are the objects of the pre-sheaf category
$\SSet=\widehat{\Delta_+}=\CAT(\Delta_+^\op,\Set)$. The nerve functor
$N\from\Cat\to\SSet$ is defined by $NC=\Cat(I\_,C)$. Typically we will
write $(NC)_i=\Cat(I\underline{i},C)$. By left Kan extension there is
a left adjoint $\tau_1\ladj{}N$; \citep[see][]{zbMATH06341429} or
\citep[see][]{zbMATH01216133}. We mention this only because this
guarantees that $N$ preserves all limits.

To be explicit, the $0$-simplices of $NC$ are the objects of $C$;
1-simplices are the arrows of $C$, $i$-simplices are strings of $i$
composable arrows. Faces are obtained by composing or dropping
elements, degeneracies result by inserting identity arrows into a string. 

There is a well-known functor, confusingly also called $\Delta$, that
goes $\widehat{\Delta_+}\to\Top$. It takes every ordinal
$\underline{n}$ to the affine $n$-simplex. Similar to the situation
with categories this defines, a functor $S\from\Top\to\SSet$ defined
as $SX=\Top(\Delta\_,X)$. Again, by Kan extension, there is a left
adjoint $|\_|\ladj{}S$ that gives the topological realization $|X|$ of
a simplicial set. Note that as a left adjoint $|\_|$ preserves all
colimits.

By combining nerve and realization one can now define the so called
classifying space functor of small categories $B=|N\_|$ introduced by
\citet{zbMATH03457106}. We shall call $BC$ the topological realization
of $C$.

The realization of an ordinal $\underline{n}$ is of course simply the
$n$-simplex $B\underline{n}$. Further simple examples can be found at
the beginning of \cref{sec:categorical-models}. A category whose
realization is $S^2$ can be found in (\cref{eq:9}).
 
\section{The Hopf Fibration}
\noindent
If we define $S^3$ as the set $\curly{(z_0,z_1)\in \C^2:
  z_0\overline{z_0}+z_1\overline{z_1}=1}$ there is a natural action
$(z_0,z_1)\mapsto(wz_0,wz_1)$ on $S^3$ by $w\in{}U(1)$. The quotient
by this action happens to be $S^2$ and the quotient map
$h\from{}S^3\to{}S^2$ is known as the Hopf map. Its homotopy class
is characterized by the fact that any two distinct fibres are simply
linked circles.

One can decompose the Hopf map by decomposing $S^2$ along the equator,
and $S^3$ along the pre-image of the equator, which is a torus. So $h$
appears as a pushout in the arrow category $\Top^\rightarrow$ of two
circle bundles over the 2-disk, see (\cref{eq:topcospan}).

\citet{zbMATH01547910} present these circle bundles as simplicial
complices. We show that instead one can describe these circle bundles
using acyclic categories and functors.

\subsection{Categorical Models}
\label{sec:categorical-models}
\noindent
We start by describing small categories that model the circle and the
disk. Let $S$ be the small category given by the following diagram:
\begin{equation}
  \label{eq:defineS}
  S=\left\{
    \begin{xy}
      *!C+\xybox{
        \xymatrix{
          A\ar@/^1pc/[r]^f\ar@/_1pc/[r]_g&B
        }
      }
    \end{xy}
  \right.
\end{equation}
with no relations. Next, the category $D$ is given by
\begin{equation}
  \label{eq:defineD}
  D=\left\{
    \begin{xy}
      *!C+\xybox{
        \xymatrix{
          A\ar@/^1pc/[r]^f\ar@/_1pc/[r]_g&B\ar[r]^t&X
        }
      }
    \end{xy}
  \right.
\end{equation}
with $tf=tg$. $S$ is obviously a full subcategory of $D$.

The central piece of our construction is the torus $T$ given the by
the diagram
\begin{equation}
  \label{eq:defineT}
  T=\left\{
    \begin{xy}
      *!C+\xybox{
        \xymatrix@-.3cm{
          A_0\ar[rr]^{f_0}\ar[ddrr]_{f_2}\ar[dd]_{p_A}&{}&B_0\ar[dr]^{p_{B_1}}&{}&A_0\ar[dd]^{p_A}\ar[ll]_{g_0}\\
          {}&{}&{}&B_1&{}\\
          A_1\ar[dr]_{f_3}&{}&B_2\ar[dl]_{p_{B_3}}\ar[ur]_{q_{B_1}}&{}&A_1\ar[ul]_{g_1}\\
          {}&B_3&{}&{}&{}\\
          A_0\ar[rr]_{f_0}\ar[uu]^{q_A}&{}&B_0\ar[ul]_{q_{B_3}}&{}&A_0\ar[uu]_{q_A}\ar[ll]^{g_0}\ar[uull]_{g_2}\\
        }
      }
    \end{xy}
  \right.
\end{equation}
with all squares commuting. There are three projection
$F_M,F_N,G\from{}T\to{}S$. Where
\begin{equation}
  \label{eq:defFM}
  \begin{aligned}
    F_M\from A_0,B_0&\mapsto{}A\\
    A_1,B_3,B_2,B_1&\mapsto{}B\\
    f_0, g_0&\mapsto\id_{Fx}\\
    f_3,p_{B_3},q_{B_1},g_1&\mapsto \id_{Fx}\\
    q_A,q_{B_3}&\mapsto{}f\\
    p_A{},p_{B_1}&\mapsto{}g
  \end{aligned}
\end{equation}
and
\begin{equation}
  \label{eq:defFN}
  \begin{aligned}
    F_N \from{}A_x&\mapsto{}A\\
    {}B_x&\mapsto{}B\\
    {}p_x,q_x&\mapsto\id_{Fx}\\
    {}f_i&\mapsto{}f\\
    {}g_i&\mapsto{}g\,.
  \end{aligned}
\end{equation}
These functors can be thought to be the vertical and the horizontal
projections of the diagram (\cref{eq:defineT}), respectively,
where the vertical and horizontal zig-zags need to be seen as
straightened in the appropriate way.

Finally, $G\from{}T\to{}S$ projects in the the top-left, bottom-right
direction. Explicitly:
\begin{equation}
  \label{eq:8}
  \begin{aligned}
    G\from{}A_0,B_2&\mapsto{}A\\
    A_1,B_3,B_0,B_1&\mapsto{}B\\
    f_3,q_{B_3},f_2,g_2,p_{B_1},g_1&\mapsto{}\id_{Fx}\\
    q_A,f_0q_{B_1}&\mapsto{}f\\
    p_A,p_{B_3},g_0&\mapsto{}g\,.
  \end{aligned}
\end{equation}

In analogy with \citet{zbMATH01547910} we define two different
presentations of the solid torus, and then identify their
boundaries. We define $M$ as the categorical mapping cylinder of
$F_M$, and $N$ as the categorical mapping cylinder of $F_N$. In
essence for a functor $F\from{}A\to{}B$ this construction introduces
an arrow $x\to{}Fx$ for each object $x$ of $A$ in a way that is
compatible with the existing arrows of $A$ and $B$.  The construction
is detailed in \cref{sec:mapp-cylind-push}.

Explicitly we get
\begin{equation}
  \label{eq:defineM}
  M=\left\{
    \begin{xy}
      *!C+\xybox{
        \xymatrix@-.3cm"T"{
          A_0\ar[rr]^{f_0}\ar[ddrr]_{f_2}\ar[dd]_{p_A}&{}&B_0\ar[dr]^{p_{B_1}}&{}&A_0\ar[dd]^{p_A}\ar[ll]_{g_0}\\
          {}&{}&{}&B_1&{}\\
          A_1\ar[dr]_{f_3}&{}&B_2\ar[dl]_{p_{B_3}}\ar[ur]_{q_{B_1}}&{}&A_1\ar[ul]_{g_1}\\
          {}&B_3&{}&{}&{}\\
          A_0\ar[rr]_{f_0}\ar[uu]^{q_A}&{}&B_0\ar[ul]_{q_{B_3}}&{}&A_0\ar[uu]_{q_A}\ar[ll]^{g_0}\ar[uull]_{g_2}\\
        }
        \POS "T1,5" +<2cm,-.7cm>
        \xymatrix"S"@ur{
          Z_0\ar[dd]\\ {}\\
          Z_1\\ {}\\
          Z_0\ar[uu]
        }
        \ar "T1,3";"S1,1"|!{"T1,5";"T5,5"**\dir{}}\hole
        \ar "T2,4";"S3,1"|!{"T1,5";"T5,5"**\dir{}}\hole
        \ar "T4,2";"S3,1"|!{"T3,3";"T5,5"**\dir{}}\hole|!{"T1,5";"T5,5"**\dir{}}\hole
        \ar "T5,3";"S5,1"
      }
    \end{xy}
  \right.
\end{equation}
and
\begin{equation}
  \label{eq:defineN}
  N=\left\{
    \begin{xy}
      *!C+\xybox{
        \xymatrix@-.3cm"T"{
          A_0\ar[rr]^{f_0}\ar[ddrr]_{f_2}\ar[dd]_{p_A}&{}&B_0\ar[dr]^{p_{B_1}}&{}&A_0\ar[dd]^{p_A}\ar[ll]_{g_0}\\
          {}&{}&{}&B_1&{}\\
          A_1\ar[dr]_{f_3}&{}&B_2\ar[dl]_{p_{B_3}}\ar[ur]_{q_{B_1}}&{}&A_1\ar[ul]_{g_1}\\
          {}&B_3&{}&{}&{}\\
          A_0\ar[rr]_{f_0}\ar[uu]^{q_A}&{}&B_0\ar[ul]_{q_{B_3}}&{}&A_0\ar[uu]_{q_A}\ar[ll]^{g_0}\ar[uull]_{g_2}\\
        }
        \POS "T1,5" +<2cm,-.7cm>
        \xymatrix"S"@ur{
          Y_0\ar[dd]\\ {}\\
          Y_1\\ {}\\
          Y_0\ar[uu]
        }
        \ar "T3,1";"S1,1"|!{"T1,1";"T3,3"**\dir{}}\hole
        |!{"T3,3";"T2,4"**\dir{}}\hole
        |!{"T2,4";"T3,5"**\dir{}}\hole
        |!{"T2,4";"S3,1"**\dir{}}\hole
        |!{"T1,5";"T5,5"**\dir{}}\hole
        \ar "T2,4";"S3,1"|!{"T1,5";"T5,5"**\dir{}}\hole
        \ar "T4,2";"S3,1"|!{"T3,3";"T5,5"**\dir{}}\hole
        |!{"T1,5";"T5,5"**\dir{}}\hole
        |!{"T3,5";"S5,1"**\dir{}}\hole
        \ar "T3,5";"S5,1"
      }
    \end{xy}
  \right.
\end{equation}
where all the squares commute.

The universal property of the strict lax pushout determines two unique
extensions of the projection $G$, $H_M\from{}M\to{}D$ and
$H_N\from{}N\to{}D$. 

These $G,H_M,H_N$ with the obvious inclusions assemble into a cospan
in the arrow category $\Cat^\rightarrow$:
\begin{equation}
  \label{eq:accatcospan}
  \begin{xy}
    *!C+\xybox{
      \xymatrix{
        {}&T\ar[r]^{}\ar[d]^{G}\ar[dl]_{}&N\ar[d]_{H_N}\ar[dr]&{}\\
        M\ar[d]_{H_M}\ar@/_1pc/[rrr]&S\ar[dl]^{}\ar[r]_{}&D\ar[dr]&P\ar[d]^H\\
        D\ar@/_1pc/[rrr]|{}&{}&{}&Q
      }
    }
  \end{xy}
\end{equation}
We define $H\from{}P\to{}Q$ as the pushout of this diagram in
$\Cat^\rightarrow$. $P$ is given by the amalgamation of the two
diagrams (\cref{eq:defineM}) and (\cref{eq:defineN}) along the
square. $Q$ is the amalgamation of two copies of (\cref{eq:defineD})
along the left-hand part; explicitly
\begin{equation}
  \label{eq:9}
  Q=\left\{
    \begin{xy}
      *!C\xybox{
        \xymatrix{
          &&Y\\
          A\ar@/^1pc/[r]^f\ar@/_1pc/[r]_{g}\ar@/^2pc/[rru]\ar@/_2pc/[rrd]&B\ar[ur]\ar[dr]&\\
          &&Z
        }
      }
    \end{xy}
  \right.\,.
\end{equation}

\subsection{Realization}
\label{sec:real-coord}
\noindent
Let $B=|N\_|$ be the topological realization functor. We apply it to
the categories described in \cref{sec:categorical-models}. It is easy
to see that we obtain $BS=S^1$, $BT=T^2$, $BD^2$, $BQ=S^2$. If we let
$T^2=([0,1]\times[0,1])/\sim$, where $(0,s)\sim(1,s)$ and
$(r,0)\sim(r,1)$, then
the realizations of the projections $H_M,H_N,G$ can be described as
$BH_M(r,s)=r$, $BH_N(r,s)=s$, $HG(r,s)=r+s$.

It is now easy to see that $BM$ is the mapping cylinder of $BF_M$, as
$BN$ is the mapping cylinder of $BF_N$.

We note that the nerve $N$ as a right adjoint can not be expected to
preserve colimits in general. Direct inspection however shows that the
pushout in (\cref{eq:accatcospan}) is preserved. $|\_|$ is a left
adjoint and preserves all colimits.

\begin{thm}
  Applying the classifying space functor $B=|N\_|$ to
  (\ref{eq:accatcospan}) gives a pushout diagram in $\Top^\rightarrow$
  \begin{equation}
    \label{eq:topcospan}
    \begin{xy}
      *!C+\xybox{\xymatrix{
          {}&T^2\ar[r]\ar[d]\ar[dl]&S^1\times{}D^2\ar[d]\ar[dr]&{}\\
          S^1\times{}D^2\ar[d]\ar@/_1pc/[rrr]&S^2\ar[dl]\ar[r]&D^2\ar[dr]&S^3\ar[d]^h\\
          D^2\ar@/_1pc/[rrr]|{}&{}&{}&S^2
        }}
    \end{xy}
  \end{equation}
  Hence $BH=h$ is the Hopf map.
\end{thm}

\begin{proof}
  $B$ preserves the mapping cylinders and pushouts since by
  \cref{lem:nerveMapPres} $N$ does so, and $|\_|$ does by being a left
  adjoint.
\end{proof}

\subsection{Mapping Cylinders and Pushouts}
\label{sec:mapp-cylind-push}

\noindent
The categorical mapping cylinder of a functor $F\from{}A\to{}B$ in
$\Cat$ is given as the strict lax pushout
\begin{equation}
  \label{eq:strLaxPoDef}
  \begin{xy}
    *!C\xybox{
      \xymatrix{
        A\ar[r]^{\Id_A}\ar[d]_{F}&A\ar[d]^{i_A}\\
        B\ar[r]_{i_B}&M_F
        \POS \ar@2 "1,2";"2,1"**\dir{}?(.4);?(.6) _{a}
      }
    }
  \end{xy}
\end{equation}
where for any 
\begin{equation}
  \label{eq:strLaxPoDef1}
  \begin{xy}
    *!C\xybox{
      \xymatrix{
        A\ar[r]^{\Id_A}\ar[d]_{F}&A\ar[d]^{H}\\
        B\ar[r]_{K}&X
        \POS \ar@2 "1,2";"2,1"**\dir{}?(.4);?(.6) _{r}
      }
    }
  \end{xy}
\end{equation}
there is a unique $G\from{}M_F\to{}X$ such that $Gi_A=H$, $Gi_B=K$ and $Ga=r$.

This means that $M_F$ has objects those of $A$ and those of $B$. The
arrows of $M_F$ are of three kinds: those in $A$, those in $B$, and
those of the form $ga_xf$ where $g\in{}B_1$, $f\in{}A_1$ (which we
shall tacitly assume for the rest of this section) and $a_x$ is the appropriate
component of the natural transformation $a$. 

Note that this is the same as the ordinary pushout
\begin{equation}
  \label{eq:strLaxPoMap}
  \begin{xy}
    *!C\xybox{
      \xymatrix{
        A\ar[d]_F\ar[r]^-{i_1}&A\times\underline{2}\ar[d]\\
        B\ar[r]&M_F
      }
    }
  \end{xy}
\end{equation}
where $\underline{2}$ is the ordinal $((01)\from{}0\to{}1)$.

The nerve $NM_F$ of $M_F$ has the following types of simplices: in
dimension 0 we have $(NM_F)_0=A_0\cup{}B_0$. In dimension 1, the
elements of $(NM_F)_1$ are of the form $f$, $g$, or $ga_xf$, where
$f\in{}A_0$, $g\in{}B_0$, $a_x\from{}x\to{}Fx$. In dimension 2
simplices are of the form $(f,f')$, $(g,g')$, $(ga_xf,f')$ or
$(g,g'a_xf)$. In dimension $n=i+j+1\geq{}3$ the simplices are of the
form $(f_1,\ldots,f_n)$, $(g_1,\ldots,g_n)$, or $(g_1,\ldots, g_{i+1}a_xf_1,\ldots,f_{j+1})$.
If $Ff_k=g_k$ for some $1\leq{}k\leq{}i+1$ we have of course
\begin{align*}
(g_1,\ldots,g_{i+1},\ldots,g_n)&=(g_1,\ldots,Ff_{i+1},\ldots,g_n)\intertext{and}
(g_1,\ldots, g_{i+1}a_xf_1,\ldots,f_{j+1})&=(g_1,\ldots, Ff_k,\ldots, g_{i+1}a_xf_1,\ldots,f_{j+1})\,.
\end{align*}

The simplicial mapping cylinder of $NF$ is of course defined as the
pushout of simplicial sets
\begin{equation}
  \label{eq:simplMapCyl}
  \begin{xy}
    *!C\xybox{
      \xymatrix{
        NA\ar[d]_{NF}\ar[r]^-{i_1}&N(A\times\underline{2})\ar[d]\\
        NB\ar[r]&M_{NF}
      }
    }
  \end{xy}\,.
\end{equation}
The 0-simplices of $M_{NF}$ are the same as those of
$NM_F$; 1-simplices are $(f,\id_0)$, $(f,(01))$, $(f,\id_1)$ or $g$
where $g=(f,\id_1)$ if $Ff=g$; 2-simplices are
$((f,\id_0),(f',\id_0))$, $(g,g')$, $((f,(01)),(f',\id_0))$ or
$((f,\id_1),(f',(01)))$.

The universal property of (\cref{eq:simplMapCyl}) induces a comparison
map $k\from{}M_{NF}\to{}NM_F$; it maps 1-simplices thus
\begin{equation}
  \label{eq:defk1}
  \begin{aligned}
    k\from{}g&\mapsto{}g\\
    (f,\id_0)&\mapsto{}f\\
    (f,\id_1)&\mapsto{}Ff\\
    (f,(01))&\mapsto{}a_xf\,.
  \end{aligned}
\end{equation}
2-simplices are mapped
\begin{equation}
  \begin{aligned}
    \label{eq:defk2}
    k\from((f,\id_0),(f',\id_0))&\mapsto(f,f')\\
    (g,g')&\mapsto(g,'g)\\
    ((f,(01)),(f',\id_0))&\mapsto(Ffa_x,f')=(a_yf,f')\\
    ((f,\id_1),(f',(01)))&\mapsto(Ff,Ff'a_x)=(Ff,a_yf')\,.\\
  \end{aligned}
\end{equation}
Higher simplices are mapped similarly
\begin{equation}
  \begin{aligned}
    \label{eq:defk3}
    k\from((f_1,\id_0),\ldots,(f_n,\id_0))&\mapsto(f_1,\ldots,f_n)\\
    (g_1,\ldots,g_n)&\mapsto(g_1,\ldots,g_n)\\
    \begin{pmatrix}
      (f_1,\id_1),\ldots,(f_i,\id_1),\\(f_{i+1},(01)),\\(f_{i+2},\id_0),\ldots,(f_n,\id_1)
    \end{pmatrix}
    &\mapsto(Ff_1,\ldots,Ff_i,Ff_{i+1}a_x,f_{i+2},\ldots,f_n)\\
    &\qquad =(Ff_1,\ldots,Ff_i,a_yf_{i+1},f_{i+2},\ldots,f_n)\,.
  \end{aligned}
\end{equation}
Note in particular that $(g,a_x)\in(NM_F)_2$ only has a pre-image
under $k$ if $g$ is in the image of $F$.

We can not always expect the nerve to preserve the structure of the
mapping cylinder: Taking for example the mapping cylinder of
$F\from{}\underline{1}\to\underline{2}$ with $F(0)=0$ we get a
category
\begin{equation}
  \label{eq:mapCylConterEx}
  M_F=\left\{
  \begin{xy}
    *!C\xybox{
      \xymatrix{
        0\ar[r]^{a_0}\ar[dr]_{(01)a_0}&F0\ar[d]^{(01)}\\
        {}&1
      }
    }
  \end{xy}\right.\,.
\end{equation}
The nerve of the mapping cylinder $NM_F$ has a 2-simplex
$((01),a_0)$, that does not exist in the simplicial mapping cylinder of
$NF$.

\begin{prp}
  \label{lem:nerveMapPres}
  The comparison map $k\from{}M_{NF}\to{}NM_F$ is an isomorphism of
  simplicial sets if the functor $F\from{}A\to{}B$ has the property
  that $g\from{}Fx\to{}y$ in $B$ implies $g=Ff$ for some $f$ in
  $A$. Hence the nerve functor $N\from\Cat\to{}\SSet$ preserves the
  mapping cylinder $M_F$.
\end{prp}

\begin{proof}
  The condition ensures that all the simplices of $NM_F$ are in the
  image of the comparison map $k$ explicitly described in
  (\cref{eq:defk1}), (\cref{eq:defk2}), and (\cref{eq:defk3}).

  Finally, $k$ is obviously injective.
\end{proof}

\begin{rem}
  \cref{lem:nerveMapPres} clearly applies to the functors $F_N,F_N,G$
  defined in \cref{sec:categorical-models}.
\end{rem}

\section{Acknowledgments}
\label{sec:aknowledgements}

\noindent
We would like to thank John Huerta for suggesting to look at the Hopf
map as an example to consider. Many thanks also are due to Benjamín
Alarcón Heredia for intense discussions.

\bibliographystyle{plainnat}
\bibliography{hopffib}

\end{document}